\documentclass[a4paper,reqno]{amsart}
\usepackage{geometry}
\usepackage{mathrsfs}
\usepackage{syntonly}
\usepackage{amsmath}
\usepackage{amsthm}
\usepackage{amsfonts}
\usepackage{amssymb}
\usepackage{latexsym}
\usepackage{amscd,amssymb,amsopn,amsmath,amsthm,graphics,amsfonts,mathrsfs,accents,enumerate,verbatim,calc}
\usepackage[dvips]{graphicx}
\usepackage[colorlinks=true,linkcolor=blue,citecolor=blue]{hyperref}
\input xy
\xyoption{all}

\usepackage{color}

\numberwithin{equation}{section}

\theoremstyle{plain}

\newtheorem{thm}{Theorem}[section]
\newtheorem{cor}[thm]{Corollary}
\newtheorem{lem}[thm]{Lemma}
\newtheorem{prop}[thm]{Proposition}

\newtheorem{defn}[thm]{Definition}
\newtheorem{exm}[thm]{Example}

\newtheorem{rem}[thm]{Remark}

\newdir{ >}{{}*!/-5pt/\dir{>}}

\newcommand{\Hom}{\operatorname{Hom}\nolimits}

\renewcommand{\Im}{\operatorname{Im}\nolimits}
\newcommand{\Ker}{\operatorname{Ker}\nolimits}
\newcommand{\Ext}{\operatorname{Ext}\nolimits}

\newcommand{\Coker}{\operatorname{Coker}\nolimits}

\newcommand{\pd}{\operatorname{pd}\nolimits}
\newcommand{\id}{\operatorname{id}\nolimits}

\newcommand{\pgl}{\operatorname{p.gl.dim}\nolimits}
\newcommand{\igl}{\operatorname{i.gl.dim}\nolimits}

\def\Id{\mathrm{Id}}

\renewcommand{\mod}{\mathsf{mod}\hspace{.01in}}
\newcommand{\Cone}{\operatorname{Cone}\nolimits}
\newcommand{\CoCone}{\operatorname{CoCone}\nolimits}

\newcommand{\B}{\mathcal B}

\newcommand{\U}{\mathcal U}
\newcommand{\V}{\mathcal V}
\newcommand{\A}{\mathcal A}

\newcommand{\s}{\mathcal S}
\newcommand{\T}{\mathcal T}

\newcommand{\D}{\mathcal D}

\newcommand{\C}{\mathcal C}

\newcommand{\E}{\mathcal E}
\newcommand{\EE}{\mathbb E}

\renewcommand{\emph}{\textit}
\renewcommand{\phi}{\varphi}

\newcommand{\add}{\mathsf{add}\hspace{.01in}}

\begin{document}

\title{Gluing $n$-tilting and $n$-cotilting subcategories}\footnote{Yu Liu was supported by the National Natural Science Foundation of China (Grant No. 11901479). Panyue Zhou was supported by the National Natural Science Foundation of China (Grant No. 11901190) and by the Scientific Research Fund of Hunan Provincial Education Department (Grant No. 19B239).}
\author{Yu Liu and Panyue Zhou}
\address{School of Mathematics, Southwest Jiaotong University, 610031 Chengdu, Sichuan, People's Republic of China}
\email{liuyu86@swjtu.edu.cn}
\address{College of Mathematics, Hunan Institute of Science and Technology, 414006 Yueyang, Hunan, People's Republic of China}
\email{panyuezhou@163.com}

\begin{abstract}
 Recently, Wang, Wei and Zhang define the recollement of extriangulated categories, which is a generalization of both recollement of abelian categories and recollement of triangulated categories. For a recollement $(\A ,\B,\C)$ of extriangulated categories, we show that
  $n$-tilting (resp. $n$-cotilting) subcategories in $\A$ and $\C$ can be glued to get $n$-tilting (resp. $n$-cotilting) subcategories in $\B$ under certain conditions. 
\end{abstract}
\maketitle

\section{Introduction}
The recollement of triangulated categories was first introduced by Beilinson, Bernstein, and Deligne \cite{BBD}. It is an important tool in algebraic geometry and representation theory.  A fundamental example of a recollement  of abelian categories appeared in the construction of perverse sheaves by MacPherson and Vilonen \cite{MV}, appearing as an inductive step in the construction.

The notion of extriangulated categories was introduced by Nakaoka and Palu \cite{NP} as a common
generalization of exact and triangulated categories.
Wang, Wei and Zhang \cite{WWZ} gave a simultaneous generalization of recollements of abelian categories and triangulated categories, which is called \emph{recollements of extriangulated categories} (see Definition \ref{recollement} for details).
A recollement of triangulated (or, abelian,  extriangulated) categories is a diagram of functors between triangulated (or, abelian,  extriangulated) categories of the following shape , which satisfies certain assumptions.

\begin{equation}\label{1-1}
   \xymatrix{\mathcal{A}\ar[rr]|{i_{*}}&&\ar@/_1pc/[ll]|{i^{*}}\ar@/^1pc/[ll]|{i^{!}}\mathcal{B}
\ar[rr]|{j^{\ast}}&&\ar@/_1pc/[ll]|{j_{!}}\ar@/^1pc/[ll]|{j_{\ast}}\mathcal{C}}
\end{equation}
\vspace{1mm}

For a recollement  $(\A,\B,\C)$ of triangulated categories,
Chen \cite{C} has described how to glue together cotorsion pairs (which are essentially equal to torsion pairs in \cite{IY}) in $\A$ and $\C$ to obtain a cotorsion pair in $\B$, which is a natural generalization of a similar result in \cite{BBD} on gluing together $t$-structures
of $\A$ and $\C$ to obtain a $t$-structure in $\B$.

The notion of cotorsion pair on extriangulated category was introduced in \cite{NP}, which is a generalization of cotorsion pair on triangulated and exact categories.
\begin{defn}\cite[Definition 4.1]{NP}
Let $\U$ and $\V$ be two subcategories of an extriangulated category $\E$. We call $(\U,\V)$ a \emph{cotorsion pair} if it satisfies the following conditions:
\begin{itemize}
\item[(a)] $\EE_\E(\U,\V)=0$.
\smallskip

\item[(b)] For any object $B\in \E$, there are two $\EE_\E$-triangles
\begin{align*}
V_B\rightarrow U_B\rightarrow B{\dashrightarrow},\quad
B\rightarrow V^B\rightarrow U^B{\dashrightarrow}
\end{align*}
satisfying $U_B,U^B\in \U$ and $V_B,V^B\in \V$.
\end{itemize}
\end{defn}

For a recollement $(\mathcal A ,\mathcal B,\mathcal C)$ of extriangulated categories,
Wang, Wei and Zhang \cite{WWZ} provided conditions such that the glued pair with respect to cotorsion pairs in $\mathcal A$ and $\mathcal C$ is still a cotorsion pair in $\mathcal B$. This result recovered a result given by Chen \cite{C} for the recollement of triangulated categories.

We provide a slightly weaker assumption on the functors in (\ref{1-1}) to get glued cotorsion pairs, which fits the recollement of abelian categories better (see Proposition \ref{key0} for details).

\begin{prop}
Let $(\A ,\B,\C)$ be a recollement of extriangulated  categories. Assume that $\B$ has enough projectives and $i^!,j_!$ are exact. Let $(\mathcal{U}_{1},\mathcal{V}_{1})$ and $(\mathcal{U}_{3},\mathcal{V}_{3})$ be cotorsion pairs in $\mathcal{A}$ and $\mathcal{C}$, respectively. Define
$$ \widetilde{\mathcal{U}_2}=\{B\in \mathcal{B}~|~i^{\ast }B\in\mathcal{U}_{1}~\text{and}~j^{\ast}B\in \mathcal{U}_{3}  \};$$
$$\mathcal{V}_2=\{B\in \mathcal{B}~|~i^{!}B\in\mathcal{V}_{1}~\text{and}~j^{\ast}B\in \mathcal{V}_{3}  \}.$$
Then $(\U_2:={^{\bot_1}}\V_2,\mathcal{V}_2)$ is a cotorsion pair in $\mathcal{B}$, where $\U_2\subseteq \add\widetilde{\U_2}$.
We call cotorsion pair $(\U_2,\V_2)$ the \emph{glued cotorsion pair }with respect to $(\U_1,\V_1)$ and $(\U_3,\V_3)$.\end{prop}

Tilting module was introduced by Brenner-Butler \cite{BB} and
Happel-Ringel \cite{HR}. Ma and Zhao \cite{MZ} gave a way of constructing a tilting module by gluing together two tilting modules in a recollement of module categories, this result also glued the correspondence torsion pairs.  The notion of $n$-tilting module was first introduced by
Miyashita \cite{M}, this coincides with the definition of tilting module when $n=1$.  An analog concept of $n$-tilting module is introduced in \cite{LZZZ}, which is called $n$-titling subcategory (see Definition \ref{deftil} for details). Dually $n$-cotilting subcategory can be defined. Every $n$-titling (resp. $n$-cotilting) subcategory admits a cotorsion pair, we call such cotorsion pair a tilting (resp. cotilting) cotorsion pair. We describe how to glue together two $n$-tilting (resp. $n$-cotilting) subcategories in $\A$ and $\C$ to obtain an  $n$-tilting (resp.  $n$-cotilting) subcategory in $\B$ for a recollement of extriangulated categories, by gluing the correspondent tilting (resp. cotiltng) cotorsion pairs (see Theorem \ref{main5} and Proposition \ref{main6} for details).

\
\begin{thm}
Assume that $(\mathcal A,\B,\C)$ is a recollement of extriangulated categories, where $\B$ has enough projectives and enough injectives, functors $i^!,j_!$ are exact.
\begin{itemize}
\item[(1)] Let $(\U_1,\V_1)$ and $(\U_3,\V_3)$ be cotilting cotorsion pairs in $\mathcal A$ and $\mathcal C$ respectively. Then the glued cotorsion pair $(\U_2,\V_2)$ in $\B$ is a cotilting cotorsion pair.
\item[(2)] Let $(\U_1,\V_1)$ and $(\U_3,\V_3)$ be tilting cotorsion pairs in $\mathcal A$ and $\mathcal C$ respectively. Assume $\mathcal A$ has finite projective global dimension, then the glued cotorsion pair $(\U_2,\V_2)$ in $\B$ is a tilting cotorsion pair.
\end{itemize}
\end{thm}

We also discuss how to glue $n$-tilting objects on abelian categories (see Theorem \ref{main7} for details).

\begin{thm}
Assume $(\mathcal A,\B,\C)$ is a recollement of abelian categories, where $\B$ has enough projectives and enough injectives, functors $i^!,j_!$ are exact. Let $(\U_1,\V_1)$ and $(\U_3,\V_3)$ be tilting cotorsion pairs in $\mathcal A$ and $\mathcal C$ respectively. Assume $\U_i\cap\V_i=\add T_i, i=1,3$. Then the glued cotorsion pair $(\U_2,\V_2)$ in $\B$ is a tilting cotorsion pair such that $\U_2\cap\V_2$ is the additive closure of an $n$-tilting object.
\end{thm}

This article is organized as follows. In Section 2, we first recall the definition and some basic properties of recollements of extriangulated categories, then we show how to glue cotorsion pairs under certain assumptions.
In Section 3, we  recall the definition of $n$-titling (resp. $n$-cotilting) subcategory in extriangulated categories and show some basic properties that we need.
In Section 4, we glue $n$-tilting and $n$-cotilting subcategories in a recollement of extriangulated categories.
In Section 5, we give some examples of our results.

\section{Preliminaries}
The definition and basic properties of extriangulated categories can be find 
in \cite[Section 2, 3]{NP}. In this article,  let $k$ be a field, $(\mathcal E,\EE_{\E},\mathfrak{s}_{\E})$ be an extriangulated category. Denote the subcategory of projective (resp. injective) objects by $\mathcal P_{\E}$ (resp. $\mathcal I_\E$).

When we say that $\C$ is a subcategory of $\E$, we always assume that $\C$ is full, closed under isomorphisms, direct sums and direct summands. Note that we do not assume any subcategory we construct has such property.

In this paper, we assume that $\E$ satisfies Condition (WIC)(\cite[Condition 5.8]{NP}):

\begin{itemize}
\item If we have a deflation $h: A\xrightarrow{~f~} B\xrightarrow{~g~} C$, then $g$ is also a deflation.
\item If we have an inflation $h: A\xrightarrow{~f~} B\xrightarrow{~g~} C$, then $f$ is also an inflation.
\end{itemize}

Note that any triangulated category and Krull-Schmidt exact category satisfies Condition {\rm (WIC)}.

\subsection{Recollement of extriangulated categories}
We recall the definition of recollement of extriangulated categories from \cite{WWZ}. We only state the settings that we need, for details, one can see \cite[Section 3]{WWZ}.


\begin{defn}\label{recollement}{\rm \cite[Definition 3.1]{WWZ}}
Let $\mathcal{A}$, $\mathcal{B}$ and $\mathcal{C}$ be three extriangulated categories. A \emph{recollement} of $\mathcal{B}$ relative to
$\mathcal{A}$ and $\mathcal{C}$, denoted by ($\mathcal{A}$, $\mathcal{B}$, $\mathcal{C}$), is a diagram
$$
  \xymatrix{\mathcal{A}\ar[rr]|{i_{*}}&&\ar@/_1pc/[ll]|{i^{*}}\ar@/^1pc/[ll]|{i^{!}}\mathcal{B}
\ar[rr]|{j^{\ast}}&&\ar@/_1pc/[ll]|{j_{!}}\ar@/^1pc/[ll]|{j_{\ast}}\mathcal{C}}
$$
given by two exact functors $i_{*},j^{\ast}$, two right exact functors $i^{\ast}$, $j_!$ and two left exact functors $i^{!}$, $j_\ast$, which satisfies the following conditions:
\begin{itemize}
  \item [(R1)] $(i^{*}, i_{\ast}, i^{!})$ and $(j_!, j^\ast, j_\ast)$ are adjoint triples.
  \item [(R2)] $\Im i_{\ast}=\Ker j^{\ast}$.
  \item [(R3)] $i_\ast$, $j_!$ and $j_\ast$ are fully faithful.
  \item [(R4)] For each $X\in\mathcal{B}$, there exists a left exact $\mathbb{E}_{\mathcal B}$-triangle sequence
  $$
  i_\ast i^! X\xrightarrow{\theta_X} X\xrightarrow{\vartheta_X} j_\ast j^\ast X\rightarrow i_\ast A
$$
  with $A\in \mathcal{A}$, where $\theta_X$ and  $\vartheta_X$ are given by the adjunction morphisms.
  \item [(R5)] For each $X\in\mathcal{B}$, there exists a right exact $\mathbb{E}_{\mathcal B}$-triangle sequence
 $$
  i_\ast A'\rightarrow j_! j^\ast X\xrightarrow{\upsilon_X} X\xrightarrow{\nu_X} i_\ast i^\ast X
$$
 with $A'\in \mathcal{A}$, where $\upsilon_X$ and $\nu_X$ are given by the adjunction morphisms.
\end{itemize}
\end{defn}

We omit the definitions of left, right exact functors and left, right exact $\EE_\E$-triangle, since they will not be used in the argument. The following remarks are useful.

\begin{rem}{\rm \cite[Proposition 3.3]{WWZ}}
An additive covariant functor $F:\mathcal A\to \B$ is exact if and only if $F$ is both left exact and right exact. Recall that $F$ is exact if for any $\EE_{\mathcal A}$-triangle $A\xrightarrow{f} B\xrightarrow{g} C\dashrightarrow$, the sequence $F(A)\xrightarrow{F(f)} F(B)\xrightarrow{F(g)} F(C)$ is an $\EE_\B$-triangle.
\end{rem}

\begin{rem}\label{77}{\rm (1)} If $\mathcal{A}$, $\mathcal{B}$ and $\mathcal{C}$ are abelian categories, then Definition \ref{recollement} coincides with the definition of recollement of abelian categories (cf. \cite{FP,P,MH}).

{\rm (2)} If $\mathcal{A}$, $\mathcal{B}$ and $\mathcal{C}$ are triangulated categories, then Definition \ref{recollement} coincides with the definition of recollement of triangulated categories (cf. \cite{BBD}).

{\rm (3)} There exist examples of recollement of an extriangulated category in which one of the categories involved is neither abelian nor triangulated, see \cite{WWZ}.
\end{rem}

We collect some properties of a recollement of extriangulated categories, which will be used in the sequel.
\begin{prop}\label{key1}\rm{\cite[Proposition 3.3]{WWZ}} Let ($\mathcal{A}$, $\mathcal{B}$, $\mathcal{C}$) be a recollement of extriangulated categories.

$(1)$ All the natural transformations
$$i^{\ast}i_{\ast}\Rightarrow\Id_{\A},~\Id_{\A}\Rightarrow i^{!}i_{\ast},~\Id_{\C}\Rightarrow j^{\ast}j_{!},~j^{\ast}j_{\ast}\Rightarrow\Id_{\C}$$
are natural isomorphisms.

$(2)$ $i^{\ast}j_!=0$ and $i^{!}j_\ast=0$.

$(3)$ $i^{\ast}$ preserves projective objects and $i^{!}$ preserves injective objects.

$(3')$ $j_{!}$ preserves projective objects and $j_{\ast}$ preserves injective objects.

$(4)$ If $i^{!}$ (resp. $j_{\ast}$) is  exact, then $i_{\ast}$ (resp. $j^{\ast}$) preserves projective objects.

$(4')$ If $i^{\ast}$ (resp. $j_{!}$) is  exact, then $i_{\ast}$ (resp. $j^{\ast}$) preserves injective objects.

$(5)$ If $\mathcal{B}$ has enough projectives, then $\mathcal{A}$ has enough projectives $\add(i^*(\mathcal P_\B))$; if $\mathcal{B}$ has enough injectives, then $\mathcal{A}$ has enough injectives $\add(i^{!}(\mathcal I_\B))$.

$(6)$  If $\mathcal{B}$ has enough projectives and $j_{\ast}$ is exact, then $\mathcal{C}$ has enough projectives $\add(j^*(\mathcal P_\B))$; if $\mathcal{B}$ has enough injectives and $j_{!}$ is exact, then $\mathcal{C}$ has enough injectives $\add(j^*(\mathcal I_\B))$.

$(7)$ If $\mathcal{B}$ has enough projectives and $i^{!}$ is  exact, then $\mathbb{E}_{\mathcal{B}}(i_{\ast}X,Y)\cong\mathbb{E}_{\mathcal{A}}(X,i^{!}Y)$ for any $X\in\mathcal{A}$ and $Y\in\mathcal{B}$.

$(7')$  If $\mathcal{C}$ has enough projectives and $j_{!}$ is  exact, then $\mathbb{E}_{\mathcal{B}}(j_{!}Z,Y)\cong\mathbb{E}_{\mathcal{C}}(Z,j^{\ast}Y)$ for any $Y\in\mathcal{B}$ and $Z\in\mathcal{C}$.

$(8)$  If $i^{\ast}$ is exact, then $j_{!}$ is  exact.

$(8')$ If $i^{!}$ is exact, then $j_{\ast}$ is exact.

\end{prop}

\begin{prop}\label{key2}\rm{\cite[Proposition 3.4]{WWZ}}
Let ($\mathcal{A}$, $\mathcal{B}$, $\mathcal{C}$) be a recollement of extriangulated categories and $X\in\mathcal{B}$. Then the following statements hold.

$(1)$ If $i^{!}$ is exact, there exists an $\mathbb{E}_\mathcal{B}$-triangle
  \begin{equation*}\label{third}
  i_\ast i^! X\rightarrow X\rightarrow j_\ast j^\ast X\dashrightarrow.
   \end{equation*}

$(2)$ If $i^{\ast}$ is exact, there exists an $\mathbb{E}_\mathcal{B}$-triangle
  \begin{equation*}\label{four}
  j_! j^\ast X\rightarrow X\rightarrow i_\ast i^\ast X \dashrightarrow.
   \end{equation*}
\end{prop}

\subsection{Gluing cotorsion pairs}

From now on, we assume all extriangulated categories are Krull-Schmidt, Hom-finite, $k$-linear. We first introduce some notions.

\begin{defn}\label{def1}
Let $\C$ and $\D$ be two subcategories in $\E$.
\begin{itemize}
\item[(a)] Denote by $\CoCone(\C,\D)$ the subcategory
$$\{X\in \E \text{ }|\text{ } \exists \text{ } \EE_{\E}\text{-triangle }  X\to C\to D\dashrightarrow ~\mbox{where}~C\in \C \text{ and }D\in \D\}.$$
Let $\Omega_\E \C=\CoCone(\mathcal P_\E,\C)$. We write an object $X$ in the form $\Omega_\E C$ if it admits an $\EE_{\E}$-triangle $X\to P\to C\dashrightarrow$ with $P\in \mathcal P_\E$. Let $\Omega^0_\E \C=\C$ and $\Omega^1_\E\C=\Omega_\E \C$. Assume we have defined $\Omega^i_\E \C$, $i\geq 1$, then we can denote $\CoCone(\mathcal P_\E, \Omega^i \C)$ by $\Omega^{i+1}_\E \C$.

\bigskip

\item[(b)] Denote by $\Cone(\C,\D)$ the subcategory
$$\{Y\in \E \text{ }|\text{ } \exists \text{ } \EE_\E\text{-triangle } C'\to D'\to Y\dashrightarrow~\mbox{where}~C'\in \C \text{ and }D'\in \D \}.$$
 Let $\Sigma_\E \D=\Cone(\D,\mathcal I_\E)$. We write an object $Y$ in the form $\Sigma_\E D$ if it admits an $\EE_\E$-triangle $D\to I\to Y\dashrightarrow$ with $I\in \mathcal I_\E$. Let $\Sigma^0_\E \D=\D$ and $\Sigma^1_\E \D=\Sigma_\E \D$. Assume we have defined $\Sigma^j_\E \D$, $j\geq 1$, then we can denote $\Cone(\Sigma^j_\E\D,\mathcal I)$ by $\Sigma^{j+1}_\E \D$.

\bigskip

\item[(c)] Let $\C^{\vee}_0=\C^{\wedge}_0=\C$. We denote $\Cone(\C^{\wedge}_{i-1},\C)$ by $\C^{\wedge}_i$ and $\CoCone(\C,\C^{\vee}_{i-1})$ by $\C^{\vee}_i$ for any $i\geq 1$.\\
We denote $\bigcup\limits_{i\geq 0}\C^{\wedge}_i$ by $\C^{\wedge}$ and $\bigcup\limits_{i\geq 0}\C^{\vee}_i$ by $\C^{\vee}$.

\smallskip

\end{itemize}
\end{defn}

In the rest of this article,  we assume that $\E$ has enough projectives and enough injectives, then we can define higher extension groups as $\EE^{i+1}_\E(X,Y ):=\EE_\E(\Omega^i_\E X,Y).$ Liu and Nakaoka \cite[Proposition 5.2]{LN} proved that
$$\EE_\E(\Omega^i_\E X,Y)\simeq \EE_\E(X,\Sigma^i_\E Y).$$
For any subcategory $\C\subseteq \E$, let
\begin{itemize}
\item[(1)] $\C^{\bot}=\{X\in \E\text{ }|\text{ }\EE^i_\E(\C,X)=0, \forall i>0\}$;
\item[(2)] $\C^{\bot_1}=\{X\in \E\text{ }|\text{ }\EE_\E(\C,X)=0\}$;
\item[(3)] ${^{\bot}}\C=\{X\in \E\text{ }|\text{ }\EE^i_\E(X,\C)=0, \forall i>0\}$;
\item[(4)] ${^{\bot_1}}\C=\{X\in \E\text{ }|\text{ }\EE_\E(X,\C)=0\}$.
\end{itemize}

\begin{defn}\cite[Definition 4.1]{NP}
Let $\U$ and $\V$ be two subcategories of $\E$. We call $(\U,\V)$ a \emph{cotorsion pair} if it satisfies the following conditions:
\begin{itemize}
\item[(a)] $\EE_\E(\U,\V)=0$.
\smallskip

\item[(b)] For any object $B\in \E$, there are two $\EE_\E$-triangles
\begin{align*}
V_B\rightarrow U_B\rightarrow B{\dashrightarrow},\quad
B\rightarrow V^B\rightarrow U^B{\dashrightarrow}
\end{align*}
satisfying $U_B,U^B\in \U$ and $V_B,V^B\in \V$.
\end{itemize}
A cotorsion pair $(\U,\V)$ is said to be hereditary if $\EE^2_\E(\U,\V)=0$.
\end{defn}

By definition, we can conclude the following result.

\begin{lem}
Let $(\U,\V)$ be a cotorsion pair in $\B$. Then
\begin{itemize}
\item[(a)] $\V=\U^{\bot_1}$;
\item[(b)] $\U={^{\bot_1}}\V$;
\item[(c)] $\U$ and $\V$ are closed under extensions;
\item[(d)] $\mathcal I_\E\subseteq \V$ and $\mathcal P_\E\subseteq \U$.
\end{itemize}
The following are equivalent for $(\U,\V)$.
\begin{itemize}
\item[\rm (1)] $\EE^2_\E(\U,\V)=0$;
\item[\rm (2)] $\EE^i_\E(\U,\V)=0$ for any $i\geq 1$;
\item[\rm (3)] $\CoCone(\U,\U)=\U$;
\item[\rm (4)] $\Cone(\V,\V)=\V$.

\end{itemize}
\end{lem}


\begin{prop}\label{key0}
Let $(\A ,\B,\C)$ be a recollement of extriangulated  categories. Assume that $\B$ has enough projectives and $i^!,j_!$ are exact. Let $(\mathcal{U}_{1},\mathcal{V}_{1})$ and $(\mathcal{U}_{3},\mathcal{V}_{3})$ be cotorsion pairs in $\mathcal{A}$ and $\mathcal{C}$, respectively. Define
$$\widetilde{\mathcal{U}_2}=\{B\in \mathcal{B}~|~i^{\ast }B\in\mathcal{U}_{1}~\text{and}~j^{\ast}B\in \mathcal{U}_{3}  \};$$
$$\mathcal{V}_2=\{B\in \mathcal{B}~|~i^{!}B\in\mathcal{V}_{1}~\text{and}~j^{\ast}B\in \mathcal{V}_{3}  \}.$$
Then $(\U_2:={^{\bot_1}}\V_2,\mathcal{V}_2)$ is a cotorsion pair in $\mathcal{B}$, where $\U_2\subseteq \add\widetilde{\U_2}$.
\end{prop}

\begin{proof}
According to the proof of \cite[Lemma 4.5(2)]{WWZ}, any object $X\in B$ admits a commutative diagram
$$\xymatrix{
X \ar[r] \ar@{=}[d] &H \ar[r] \ar[d] &j_!U_3 \ar@{-->}[r] \ar[d] &\\
X \ar[r] &V \ar[d] \ar[r] &U \ar@{-->}[r] \ar[d] &\\
&i_*U_1 \ar@{=}[r] \ar@{-->}[d] &i_*U_1 \ar@{-->}[d]\\
&&
}
$$
where $V\in \V_2$, $U_1\in \U_1$ and $U_3\in \U_3$. By applying $\Hom_B(-,\V_2)$ to the third column, we can get an exact sequence
$$\EE_\B(i_*U_1,\V_2)\to \EE_\B(U,\V_2)\to \EE_\B(j_!U_3,\V_2).$$
By Proposition \ref{key1}, we have  $\EE_\B(i_*U_1,\V_2)\simeq \EE_\mathcal A(U_1,i^!\V_2)=0$ and $\EE_\B(j_!U_3,\V_2)\simeq \EE_\C(U_3,j^*\V_2)=0$. Hence  $\EE_\B(U,\V_2)=0$, which implies $U\in \U_2$.

Since $\B$ has enough projectives, $X$ admits an $\EE_\B$-triangle $\Omega_\B X\to P_X\to X\dashrightarrow$ with $P_X\in \mathcal P_\B$. Since $\Omega_\B X$ admits an $\EE_\B$-triangle $\Omega_\B X\to V_2\to U_2\dashrightarrow$ where $V_2\in \V_2$ and $U_2\in \U_2$, we can get the following commutative diagram.
$$\xymatrix{
\Omega_\B X \ar[r] \ar[d] &P_X \ar[r] \ar[d] &X \ar@{=}[d] \ar@{-->}[r] &\\
V_2 \ar[r] \ar[d] &U_2' \ar[d] \ar[r] &X  \ar@{-->}[r] &\\
U_2 \ar@{=}[r] \ar@{-->}[d] &U_2 \ar@{-->}[d]\\
&&
}
$$
Since $P_X\in \U_2$, we have $U_2'\in \U_2$.  To get that $(\U_2,\V_2)$ is a cotorsion pair, we still need to show that $\V_2$ is closed under direct sums and direct summands. It is enough to show that for any object $Y$, if $\EE_\B(X,Y)=0$ with $X\in \U_2$, then $Y\in \V_2$.

Let $U_1\in \U_1$, then $i_*U_1\in \U_2$ since $\EE_\B(i_*U_1,\V_2)\simeq \EE_\mathcal A(U_1,i^!\V_2)=0$. We have $\EE_\mathcal A(U_1,i_!Y)\simeq \EE_\B(i_*U_1,Y)=0$. Hence $i^!Y\in \V_1$. By the similar method, we can get that $j^*Y\in \V_3$. Hence $Y\in \V_2$.

Let $U_2'$ be any object in $\U_2$. By \cite[Lemma 4.5(1)]{WWZ}, $U_2'$ admits an $\EE_\B$-triangle $V\to U\to U'_2\dashrightarrow$ where $V\in \V_2$ and $U\in \widetilde{\U_2}$. Since $\EE_\B(U_2',V)=0$, this sequence splits, which implies that $U_2'$ is a direct summand of $U$. Hence $\U_2\subseteq \add\widetilde{\U_2}$.
\end{proof}

\begin{rem}
Note that by Proposition \ref{key1}, $i^*$ is exact implies that $j_!$ is exact. If we assume the exactness of $i^*$ instead of $j_!$, we can get that $\widetilde{\U_2}=\U_2$.
\end{rem}

Under the same settings as in Proposition \ref{key0}, we show the following proposition.

\begin{prop}\label{p1}
Assume $\B$ also has enough injectives. If $(\U_1,\V_1)$ and $(\U_3,\V_3)$ are hereditary, then $(\U_2,\V_2)$ is also hereditary.
\end{prop}

\begin{proof}
We only need to show that $\CoCone(\V_2,\V_2)\subseteq \V_2$. Let
$$V\to V'\to X\dashrightarrow$$
be an $\EE_\B$-triangle with $V,V'\in \V_2$. Since $i^!$ is exact, we get an $\EE_\mathcal A$-triangle
$$i^!V\to i^!V'\to i^!X\dashrightarrow.$$
By definition of $\V_2$, we have $i^!V, i^!V'\in \V_1$. Since $(\U_1,\V_1)$ is hereditary, we have $\CoCone(\V_1,\V_1)= \V_1$, hence $i^!X\in \V_1$. By the similar argument we can show that $j^*X\in \V_3$. Hence by definition $X\in \V_2$.
\end{proof}

We call the cotorsoin pair $(\U_2,\V_2)$ got in Proposition \ref{key0} the \emph{glued cotorsion pair with respect to $(\U_1,\V_1)$ and $(\U_3,\V_3)$}. We have the following observation.
\begin{prop}
Let $(\A ,\B,\C)$ be a recollement of abelian categories. Assume that $i^!,i^*$ are exact, $\B$ has enough projectives and enough injectives . Let $(\mathcal{U}_{1},\mathcal{V}_{1})$ and $(\mathcal{U}_{3},\mathcal{V}_{3})$ be cotorsion pairs in $\mathcal{A}$ and $\mathcal{C}$ respectively. Let $(\U_2,\V_2)$ be the glued cotorsion pair and $\T_i=\U_i\cap \V_i, i=1,2,3$. Then any indecomposable object $T\in \T_2$ satisfies one of the following conditions:
\begin{itemize}
\item[(1)] There is an indecomposable object $T'\in \T_1$ such that $T\simeq i_*T'$.
\item[(2)] There is an indecomposable object $T''\in \T_3$ such that $T\simeq j_!T''$.
\end{itemize}
\end{prop}

\begin{proof}
Let $T\in \T_2$ be any indecomposable object. It admits an $\EE_\B$-triangle
$$j_! j^* T\rightarrow T\rightarrow i_\ast i^\ast T\dashrightarrow.$$
We have $i^* T\in \U_1$ and $i_\ast i^* T\in\U_2$. We also have $j^\ast T\in \T_3$. Since $i^!j_!=0$ when $i^*$ is exact in the recollement of abelian categories, we have $j_!j^*T\in \T_2$. Thus this sequence splits and we get $T\simeq j_! j^\ast T$ or $T\simeq i_\ast i^* T$.

If $T\simeq j_! j^\ast T$, since $j_!$ is faithful, we can get that $j^*T$ is indecomposable. Hence condition {\rm (2)} is satisfied.

Assume $T\simeq i_\ast i^* T$. $i_*T$ admits an $\EE_\mathcal A$-triangle $i^*T\to T_1\to U_1\dashrightarrow$ where $T_1\in \T_1$ and $U_1\in \U_1$. By applying $i_*$, we can get an $\EE_\B$-triangle $T\to i_*T_1\to i_*U_1\dashrightarrow$. Since $i_*U_1\in \U_2$, this sequence splits. Hence $T$ is a direct summand of $i_*T_1$, which implies that $i^*T$ is a direct summand of $T_1$. Since $i_*$ is faithful, we can get that $i^*T$ is indecomposable. Hence the condition {\rm (1)} is satisfied.
\end{proof}

We can get the following corollary immediately.

\begin{cor}
Under the settings of the previous proposition, assume $\T_1=\add T_1$ and $\T_3=\add T_3$, then $\T_2=\add T_2$ where $T_2=i_*T_1\oplus j_!T_3$.
\end{cor}

\section{$n$-tilting and $n$-cotilting subcategories}

\begin{defn}
A subcategory $\D\subsetneq \E$ is said to have finite projective dimension if there is a natural number $n$ such that $\D\subseteq (\mathcal P_\E)^{\wedge}_n$. The minimal $n$ that satisfies this condition is called the projective dimension of $\D$. In this case, we write $\pd_\E \D=n$. The projective dimension of an object $D$ is just the projective dimension of $\add D$.

Dually, we can define the injective dimension $\id_\E \D$ (resp. $\id_\E D$)of a subcategory $\D$ (resp. an object $D$).
\end{defn}

\begin{defn}\label{deftil}\cite[Defintion 3.2]{LZZZ}
A subcategory $\T\subseteq \E$ is called an $n$-tilting subcategory~{\rm (}$n\geq 1${\rm )} if
\begin{itemize}
\item[(P1)] $\pd_\E \T\leq n$.
\item[(P2)] $\EE^i_\E(\T,\T)=0,i=1,2,...,n$.
\item[(P3)] Any projective object $P$ admits $\EE_\E$-triangles
$$P\to T_0\to R_1\dashrightarrow,\quad R_1\to T_1\to R_2\dashrightarrow,\quad \cdots, R_{n-1}\to T_{n-1}\to T_n\dashrightarrow$$
where $T_i\in \T,i=0,1,...,n$.
\end{itemize}
For convenience, any $n$-tilting subcategory can be simply called a generalized tilting subcategory. An object $T$ is called a generalized tilting object if $\add T$ is a generalized tilting subcategory. Dually we can define $n$-cotilting subcategory and $n$-cotilting object.
\end{defn}

\begin{rem}
According to \cite[Remark 4]{ZhZ}, the tilting subcategory of projective dimension $n$ defined in \cite[Definition 7]{ZhZ} is a special case of $n$-tilting subcategory in Definition \ref{deftil}.
\end{rem}

The following results are useful.

\begin{lem}\label{lem1}{\rm (see \cite[Section 3]{AT} for details)}
Let $\s$ be a subcategory in $\E$ such that $\EE_\E^i(\s,\s)=0,\forall i>0$. Then
\begin{itemize}
\item[(1)] $\EE^i_{\E}(\s^{\vee},\s^{\wedge})=0, \forall i>0$ and $\s^{\vee}\cap \s^{\wedge}=\s$.
\item[(2)] $\s^{\vee}$ is closed under direct sums, direct summands and extensions. $\CoCone(\s^{\vee},\s^{\vee})=\s^{\vee}$.
\item[(3)] $\s^{\wedge}$ is closed under direct sums, direct summands and extensions. $\Cone(\s^{\wedge},\s^{\wedge})=\s^{\wedge}$.
\item[(4)] $\s^{\bot}=(\s^{\vee})^{\bot}$ and ${^{\bot}}\s=({^{\bot}}\s)^{\wedge}$.
\end{itemize}
\end{lem}

\begin{prop}\label{main3}\cite[Proposition 3.7]{LZZZ}
 Let $\T\subseteq \B$ such that $\EE^i_\E(\T,\T)=0, \forall i>0$ and $\pd_\E \T\leq n$. Consider the following conditions:
\begin{itemize}
\item[(a)] $\T$ is contravariantly finite and $n$-tilting;
\item[(b)] $(\T^{\vee},\T^{\bot})$ is a cotorsion pair;
\item[(c)] $\T$ is an $n$-tilting subcategory.
\end{itemize}
We have {\rm  (a)$\Rightarrow$(b)$\Rightarrow$(c).}
\end{prop}

\begin{defn}
We call a hereditary cotorsion pair $(\U,\V)$ a \emph{tilting cotorsion pair} if the following conditions are satisfied:
\begin{itemize}
\item[(1)] $\T=\U\cap \V$ is a generalized tilting subcategory.
\item[(2)] $\U=\T^{\vee}$ and $\V=\T^{\bot}$.
\end{itemize}
Dually we can define \emph{cotilting cotorsion pair.}
\end{defn}


\begin{prop}\label{key3}
Let $(\U,\V)$ be a hereditary cotorsion pair. Then $(\U,\V)$ is a tilting cotorsion pair if and only if If $\pd_\E\U<\infty$.
\end{prop}

\begin{proof}
Let $\T=\U\cap \V$. Assume that $\pd_\E\U\leq n$, then $\EE^i_\E(\U,\U)=0, \forall i>n$. For any object $U\in \U$, we have $\EE_\E$-triangles $U_{i-1}\to T_i\to U_i,i=1,...,n+1$, $T_i\in \T,U_i\in \U, U_0=U$. By applying $\Hom_{\E}(\U,-)$ to these $\EE_\E$-triangles, we can get the following exact sequences
$$0=\EE^{i}_\E(\U,T_{n-i+2})\to \EE^{i}_\E(\U,U_{n-i+2})\to \EE^{i+1}_\E(\U,U_{n-i+1})\to \EE^{i+1}_\E(\U,T_{n-i+2})=0$$
with $i=1,....,n.$ Hence we have $\EE^{i}_\E(\U,U_{n-i+2})\simeq \EE^{i+1}_\E(\U,U_{n-i+1}), i=1,....,n$. But $\EE^{n+1}_\E(\U,U_1)=0$, we get that $\EE_\E(\U,U_{n+1})=0$.
Hence $U_{n+1}\in \T$. This implies $\U\subseteq \T^{\vee}_n$. Since $(\U,\V)$ is a hereditary cotorsion pair, we have $\T^{\vee}\subseteq \U$. This shows that $\U=\T^{\vee}$. We have $(\T^{\vee})^{\bot_1}\supseteq (\T^{\vee})^{\bot}=\T^{\bot}$. On the other hand, $\Omega^i_\E\T\subseteq \T^{\vee}, \forall i>0$, then $X\in (\T^{\vee})^{\bot_1}$ implies that $\EE_\E(\Omega^i\T,X)=\EE^i_\E(\T,X)=0, \forall i>0$. Hence $\V=\T^{\bot}$.

Now let $(\U,\V)$ be a tilting cotorsion pair. Then $\U=\T^{\vee}$. Let $\pd_\E \T\leq n$. We can easily get that for each object $U\in \T^{\vee}$, $\pd_\E U\leq n$. Hence $\pd_\E\U\leq n$.
\end{proof}

From the proof this proposition, we can easily get the following corollary.

\begin{cor}\label{key3cor}
Let $(\U,\V)$ be a hereditary cotorsion pair and $\T=\U\cap \V$. If $\pd_\E \U\leq n$, then $\U=\T^{\vee}_n$.
\end{cor}

\medskip

\section{Gluing $n$-tilting and $n$-cotilting subcategories}

\begin{defn}
$\E$ is said to have finite projective global dimension if there is a natural number $n$ such that $\E=(\mathcal P_\E)^{\wedge}_n$. The minimal $n$ that satisfies this condition is called the global dimension of $\E$. In this case, we write $\pgl \E=n$.

Dually we can define the injective global dimension $\igl\E$ of $\E$.
\end{defn}

In this section, we always assume the following:
\begin{itemize}
\item[(a)] $(\mathcal A,\B,\C)$ is a recollement of extriangulated categories.
\item[(b)] $\B$ has enough projectives and enough injectives.
\item[(c)] $i^!,j_!$ are exact.
\item[(d)] $(\U_1,\V_1)$ and $(\U_3,\V_3)$ are hereditary pairs in $\mathcal A$ and $\C$ respectively. $(\U_2,\V_2)$ is the glued cotorsion pair in $\B$. Denote $\U_i\cap\V_i$ by $\T_i, i=1,2,3$.
\end{itemize}

The following corollary is a direct conclusion of Proposition \ref{key3}.

\begin{cor}
 If $\pgl \B<\infty$ (resp. $\igl \B<\infty$), then $(\U_2,\V_2)$ is a tilting (resp. cotilting) cotorsion pair.
\end{cor}


\subsection{Gluing $n$-tilting subcategories}

\begin{prop}\label{main6}
Let $(\U_1,\V_1)$ and $(\U_3,\V_3)$ be tilting cotorsion pairs. Assume $\pgl \mathcal A<\infty$, then $(\U_2,\V_2)$ is a tilting cotorsion pair.
\end{prop}

\begin{proof}
By Proposition \ref{key3} and Proposition \ref{p1}, it is enough to check that $\pd_\B \U_2<\infty$. Assume $\pgl \mathcal A\leq n_1$ and $\pd_\C\U_3\leq n_3$.

Let $U\in \U_2$ be any object. By Proposition \ref{key0}, there exists an object $U'\in \add \widetilde{\U_2}$ such that $\widetilde{U}=:U'\oplus U$ satisfies $i^*\widetilde{U}\in \U_1$ and $j^*\widetilde{U}\in \U_3$. There exists a right exact $\mathbb{E}_{\mathcal B}$-triangle sequence
 $$
  i_\ast A\rightarrow j_! j^\ast \widetilde{U}\rightarrow \widetilde{U}\rightarrow i_\ast i^\ast \widetilde{U}
$$
 with $A\in \mathcal{A}$, which is, in fact, a combination of two $\EE_\B$-triangles
$$i_\ast A\rightarrow j_! j^\ast \widetilde{U}\rightarrow X\dashrightarrow, \quad X\rightarrow \widetilde{U}\rightarrow i_\ast i^\ast \widetilde{U}\dashrightarrow.$$
Since $\pgl \mathcal A\leq n_1$, we have $\pd_\mathcal A A\leq n_1$. Since $i_*$ preserves projectives, we have $\pd_\B i_*A\leq n_1$. Since $j^*\widetilde{U}\in \U_3$, we have $\pd_\C j^*\widetilde{U}\leq n_3$. Since $j_!$ preserves projectives, we have $\pd_\B (j_!j^*\widetilde{U})\leq n_3$. Hence $\pd_\B X\leq \max\{n_1+1,n_3\}$. Since $i^*\widetilde{U}\in \U_1$, we have $\pd_\mathcal A i^*\widetilde{U}\leq n_1$, then $\pd_\B (i_*i^*\widetilde{U})\leq n_1$. Hence $\pd_\B \widetilde{U}\leq \max\{n_1+1,n_3\}$. This implies $\pd_\B U\leq \max\{n_1+1,n_3\}$. Thus $\pd_\B\U_2\leq \max\{n_1+1,n_3\}$ and by Proposition \ref{key3}, $(\U_2,\V_2)$ is a tilting cotorsion pair.
\end{proof}

When we glue tilting objects on abelian categories, we can drop the assumption of projective dimension finiteness. We need some preparation.

First note that $i_*\T_1\subseteq \T_2$.

For any object $T''\in \T_3$, we have $j_!T''\in \U_2$. It admits an $\EE_\B$-triangle
$$i_*i^!j_!T''\to j_!T''\to j_*T''\dashrightarrow.$$
Since $i^!j_!T''\in \mathcal A$, it admits an $\EE_{\mathcal A}$-triangle
$$i^!j_!T''\to V_1\to U_1\dashrightarrow$$
where $U_1\in \U_1$ and $V_1\in \V_1$. Then we have an $\EE_\B$-triangle
$$i_*i^!j_!T''\to i_*V_1\to i_*U_1\dashrightarrow$$
where $i_*U_1\in \U_2$ and $i_*V_1\in \V_2$. Now we have the following commutative diagram of $\EE_\B$-triangles.
$$\xymatrix{
i_*i^!j_!T''\ar[r] \ar[d] &j_!T''\ar[r]  \ar[d] &j_*T'' \ar@{=}[d] \ar@{-->}[r] & \ar@{}[d]^{(\bigstar)}\\
i_*V_1 \ar[r] \ar[d] &K_{T''} \ar[d] \ar[r] &j_*T''  \ar@{-->}[r] &\\
i_*U_1 \ar@{=}[r] \ar@{-->}[d] &i_*U_1 \ar@{-->}[d]\\
&&
}
$$
Since $j_!T''\in \U_2$, $j_*T''\in \V_2$, we have $K_{T''}\in \T_2$.

Assume $\T_3=\add T_3$ such that $T_3=\bigoplus\limits_{i=1}\limits^n T_3^i$, where $T^i_3$ are indecomposable objects. Let $K_{T_3^i}$ be the object got in the above diagram with respect to $T_3^i$. Denote $\bigoplus\limits_{i=1}\limits^n K_{T_3^i}$ by $K_{T_3}$.

\begin{thm}\label{main7}
Let $(\mathcal A,\B,\C)$ be a recollement of abelian categories. Let $(\U_1,\V_1)$ and $(\U_3,\V_3)$ be tilting cotorsion pairs. Assume $\T_i=\add T_i,i=1,3$, then $(\U_2,\V_2)$ is a tilting cotorsion pair such that $\T_2=\add (i_*T_1\oplus K_{T_3})$.
\end{thm}

\begin{proof}
Denote $i_*T_1\oplus K_{T_3}=$ by $T_2'$ and $\add T_2'$ by $\T_2'$. Assume $\pd_\mathcal A \U_1\leq n_1$ and $\pd_\C\U_3\leq n_3$.

Since $T_2'\in \T_2$, we have $\Ext^i_{\B}(T_2',T_2')=0$. Definition \ref{deftil}{\rm (P2)} is satisfied.

Since $T_1\in \U_1$ and $\pd_\mathcal A \U_1\leq n_1$, we have $\pd_\B i_*T_1\leq n_1$. Since $T_3\in \U_3$ and $\pd_\mathcal C \U_3\leq n_3$, according to diagram $(\bigstar)$, we have $\pd_\B K_{T_3}\leq \max\{n_1,n_3\}$. Hence $\pd_\B T_2'\leq \max\{n_1,n_3\}$. Definition \ref{deftil}{\rm (P1)} is satisfied.

Let $P\in \mathcal P_\B$ be an indecomposable object. It admits two short exact sequences
$$i_\ast A\rightarrowtail j_! j^\ast P\twoheadrightarrow X, \quad X\rightarrowtail P\twoheadrightarrow i_\ast i^\ast P$$
with $A\in \mathcal A$. Since $i^*$ and $i_*$ preserve projectives, we have $i_\ast i^\ast P\in \mathcal P_\B$. Hence the second sequence splits. Since $P$ is indecomposable, we can get that $P\simeq i_*i^*P$ or $P\simeq X$. The second case implies that $P$ is a direct summand of $j_! j^\ast P$. Since $j^*$ preserves projectives, we get that any indecomposable object $P\in \mathcal P_\B$ satisfies one of the following conditions:
\begin{itemize}
\item[(1)] There is an object $P_1\in \mathcal P_\mathcal A$ such that $P\simeq i_*P_1$.
\item[(2)] There is an object $P_3\in \mathcal P_\C$ such that $P\simeq j_!P_3$.
\end{itemize}
If $P\simeq i_*P_1$, since $P_1\in \U_1=(\T_1)^{\vee}_{n_1}$, we have $P\in (i_*\T_1)^{\vee}_{n_1}\subseteq (\T_2')^{\vee}_{n_1}$.\\
If $P\simeq j_!P_3$, since $P_3\in \U_3=(\T_3)^{\vee}_{n_3}$, we have the following short exact sequences
$$P\rightarrowtail j_!T_3^0 \twoheadrightarrow j_!U_3^1,\quad j_!U_3^1 \rightarrowtail j_!T_3^1 \twoheadrightarrow j_!U_3^2,\quad \cdots, \quad j_!U_3^{n_3-1} \rightarrowtail j_!T_3^{n_3-1} \twoheadrightarrow j_!T_3^{n_3}$$
where $T_3^i\in \T_3$ and $U_3^j\in \U_3$. Since $j_!T_3^0$ admits a short exact sequence $j_!T_3^0 \rightarrowtail K_{T_3^0}\twoheadrightarrow i_*U_1^0$ where $K_{T_3^0}\in \add (K_{T_3})$ and $U_1^0\in \U_1$, we get the following commutative diagram.
$$\xymatrix{
P \ar@{ >->}[r] \ar@{=}[d] &j_!T_3^0 \ar@{ >->}[r] \ar@{->>}[d] &j_!U_3^1 \ar@{ >->}[d]\\
P \ar@{ >->}[r] &K_{T_3^0} \ar@{->>}[r] \ar@{->>}[d] &U_2^1 \ar@{->>}[d]\\
&i_*U_1^0 \ar@{=}[r] &i_*U_1^0
}
$$
Since $j_!U_3^1,i_*U_1^0 \in \U_2$, we have $U_2^1\in \U_2$. $j_!U_3^1$ admits the following commutative diagram
$$\xymatrix{
j_!U_3^1 \ar@{ >->}[r] \ar@{=}[d] &j_!T_3^1 \ar@{ >->}[r] \ar@{->>}[d] &j_!U_3^2 \ar@{ >->}[d]\\
j_!U_3^1 \ar@{ >->}[r] &K_{T_3^1} \ar@{->>}[r] \ar@{->>}[d] &U_2^2 \ar@{->>}[d]\\
&i_*U_1^1 \ar@{=}[r] &i_*U_1^1
}
$$
where $K_{T_3^1}\in \add (K_{T_3})$ and $U_1^1\in \U_1$. This implies $U_2^2\in \U_2$. Since $i_*U_1^0$ admits a short exact sequence $i_*U_1^0\rightarrowtail i_*T_1^1\twoheadrightarrow i_*(U_1^1)'$ where $T^1_1\in \T_1$ and $(U_1^1)'\in \U_1$, we have the following commutative diagram
$$\xymatrix{
j_!U_3^1 \ar@{ >->}[r] \ar@{ >->}[d] &K_{T_3^1} \ar@{->>}[r] \ar@{ >->}[d] &U_2^2 \ar@{ >->}[d]\\
U_2^1  \ar@{ >->}[r] \ar@{->>}[d] &K_{T_3^1}\oplus i_*T_1^1 \ar@{->>}[r] \ar@{->>}[d] & U_2^2\oplus i_*(U_1^1)' \ar@{->>}[d]\\
i_*U_1^0 \ar@{ >->}[r] &i_*T_1^1 \ar@{->>}[r] &i_*(U_1^1)'
}
$$
where $K_{T_3^1}\oplus i_*T_1^1\in \add T_2'$. Now we only need to focus on $U^2_2$. Since it admits a short exact sequence $j_!U_3^2 \rightarrowtail U_2^2\twoheadrightarrow i_*U_1^1$, we can continue this process and get the following exact sequences:
$$P\rightarrowtail \widetilde{T_2^0} \twoheadrightarrow U_2^1,\quad U_2^1 \rightarrowtail \widetilde{T_2^1} \twoheadrightarrow U_2^2\oplus i_*(U_1^1)'\quad \cdots, \quad U_2^{n_3-1} \rightarrowtail \widetilde{T_2^{n_3}} \twoheadrightarrow U_2^{n_3}\oplus i_*(U_1^{n_3-1})'$$
where $\widetilde{T_2^i}\in \T_2'$, $U_2^j\in \U_2$ and $(U_1^{k})'\in \U_1$. Moreover, $U_2^{n_3}$ admits a short exact sequence
$$j_!T_3^{n_3}\rightarrowtail U_2^{n_3}\twoheadrightarrow i_*U_1^{n_3-1}.$$
Then we have the following commutative diagram.
$$\xymatrix{
j_!T_3^{n_3} \ar@{ >->}[r] \ar@{ >->}[d] &K_{T_3^{n_3}} \ar@{->>}[r] \ar@{ >->}[d] &i_*U_1^{n_3} \ar@{ >->}[d]\\
U_2^{n_3}  \ar@{ >->}[r] \ar@{->>}[d] &K_{T_3^{n_3}}\oplus i_*T_1^{n_3} \ar@{->>}[r] \ar@{->>}[d] & i_*U_1^{n_3}\oplus i_*(U_1^{n_3})' \ar@{->>}[d]\\
i_*U_1^{n_3-1} \ar@{ >->}[r] &i_*T_1^{n_3} \ar@{->>}[r] &i_*(U_1^{n_3})'
}
$$
Hence we get that $P\in (\T_2')^{\vee}_{(n_1+n_3+1)}$,  Definition \ref{deftil}{\rm (P3)} is satisfied.

Note that the argument above also shows that $j_!\U_3\subseteq  (\T_2')^{\vee}$.

By Proposition \ref{main3}, $((\T_2')^{\vee}, (\T_2')^{\bot})$ is a cotorsion pair. Since $T_2'\in \T_2$, we have $(\T_2')^{\bot}\supseteq \V_2$. Let $X\in (\T_2')^{\bot}$. We show that $X\in \V_2$.

$\Ext^1_{\C}(\U_3,j^*X)\simeq \Ext^1_{\B}(j_!\U_3,X)$. Since $j_!\U_3\subseteq  (\T_2')^{\vee}$ and $X\in (\T_2')^{\bot}=((\T_2')^{\vee})^{\bot}$, we have $\Ext^1_\B(j_!\U_3,X)=0$. Hence $j^*X\in\V_3$.

$\Ext^1_{\mathcal A}(\U_1,i^!X)\simeq \Ext^1_{\B}(i_*\U_1,X)$. Since $i_*\U_1\subseteq  (\T_2')^{\vee}$, we have $\Ext^1_\B(i_*\U_1,X)=0$. Hence $i^!X\in\V_1$. This means $X\in \V_2$.

Thus $(\U_2,\V_2)=((\T_2')^{\vee}, (\T_2')^{\bot})$ is a tilting cotorsion pair and $\T_2'=\T_2$.
\end{proof}

\subsection{Gluing $n$-cotilting subcategories}

\begin{thm}\label{main5}
Let $(\U_1,\V_1)$ and $(\U_3,\V_3)$ be cotilting cotorsion pairs. Then $(\U_2,\V_2)$ is a cotilting cotorsion pair.
\end{thm}

\begin{proof}
By Proposition \ref{p1} and the dual of Proposition \ref{key3}, it is enough to show that $\id_\B \V_2<\infty$. By the dual of Proposition \ref{key3}, we can assume that $\id_\mathcal A\V_1\leq n_1$ and $\id_\mathcal C\V_3\leq n_3$.

For any $I\in \mathcal I_\B$, since $i^!$ is exact, $I$ admits an $\EE_\B$-triangle
$$i_*i^!I\to I\to j_*j^*I\dashrightarrow.$$
Since $j_*,j^*$ preserves injectives, we have $j_*j^*I\in \mathcal I_\B$. This implies $\id_\B(i_*i^!I)\leq 1$.

For any $V\in \V_2$, we have an $\EE_\B$-triangle
$$i_*i^!V\to V\to j_*j^*V\dashrightarrow.$$
We have $j^*V\in \V_3$, hence $\id_\B(j_*j^*V)\leq n_3$. We also have $i^!V\in \V_1$, then $\id_\mathcal A i^!V\leq n_1$. By Proposition \ref{key1} (5), $\mathcal A$ has enough injectives $\add(i^!\mathcal I_\B)$, we have the following $\EE_\mathcal A$-triangles.
$$i^!V \to I_1^1 \to R_1\dashrightarrow,\quad R_1\to I_1^2\to R_2\dashrightarrow,\quad \cdots,\quad R_{n_1}\to I_1^{n_1}\to I_1^{n_1+1}\dashrightarrow$$
where $I_1^j\in \add(i^!\mathcal I_\B), j=1,2,...,n_1+1$. Then $i_*I_1^j$ is a direct summand of some object in $i_*i^!\mathcal I_\B$, which implies that $\id_\B(i_*I_1^j)\leq 1, j=1,2,...,n_1+1$. Since $i_*$ is exact, we can get that $\id_\B(i_*i^!V)\leq n_1+1$. Hence $\id_\B\V_2\leq\max\{n_1+1,n_3 \}$.
\end{proof}

\section{Examples}

In this section we give some examples of our results.

Let $\Lambda',\Lambda''$ be artin algebras and $_{\Lambda'}N_{\Lambda''}$ an $(\Lambda',\Lambda'')$-bimodule, and let $\left(\begin{smallmatrix}
                \Lambda'&N\\
                0&\Lambda''
              \end{smallmatrix}
            \right)$
be a triangular matrix algebra.
Then any module in $\mod \Lambda$ can be uniquely written as a triple ${X\choose Y}_{f}$ with $X\in\mod \Lambda'$, $Y\in\mod \Lambda'$
and $f\in\Hom_{\Lambda'}(N\otimes_{\Lambda''}Y,X)$, see \cite[page 76]{ARS}.

\begin{exm}\label{example}
Let $\Lambda'$ be the finite dimensional algebra given by the quiver $\xymatrix@C=15pt{1\ar[r]&2}$ and $\Lambda''$ be the finite dimensional algebra given by the quiver $\xymatrix@C=15pt{3\ar[r]^{\alpha}&4\ar[r]^{\beta}&5}$ with the relation $\beta\alpha=0$. Define a triangular matrix algebra $\Lambda={\Lambda'\ \Lambda'\choose \ 0\ \ \Lambda''}$, where the right $\Lambda''$-module structure on $\Lambda'$ is induced by the unique algebra surjective homomorphsim $\xymatrix@C=15pt{\Lambda''\ar[r]^{\phi}&\Lambda'}$ satisfying $\phi(e_{3})=e_{1}$, $\phi(e_{4})=e_{2}$, $\phi(e_{5})=0$.  Then $\Lambda$ is
a finite dimensional algebra given by the quiver
$$
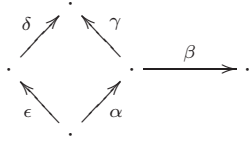
\xymatrix@C=0.5cm@R0.5cm{&\cdot\\
\cdot\ar[ru]^{\delta}&&\ar[lu]_{\gamma}\cdot\ar[rr]^-{\beta}&&\cdot\\
&\ar[lu]^{\epsilon}\cdot\ar[ru]_{\alpha}}$$
with the relation $\gamma\alpha=\delta\epsilon$ and $\beta\alpha=0$. The Auslander-Reiten quiver of $\Lambda$ is
$$\xymatrix@C=15pt{{0\choose P(5)}\ar[rd]&&{S(2)\choose S(4)}\ar[rd]&&{S(1)\choose 0}\ar[rd]&&0\choose P(3)\ar[rd]\\
&{S(2)\choose P(4)}\ar[ru]\ar[rd]&&P(1)\choose S(4)\ar[ru]\ar[r]\ar[rd]&P(1)\choose P(3)\ar[r]&S(1)\choose P(3)\ar[ru]\ar[rd]&&{0\choose S(3)}.\\
S(2)\choose 0\ar[ru]\ar[rd]&&P(1)\choose P(4)\ar[ru]\ar[rd]&&0\choose S(4)\ar[ru]&&S(1)\choose S(3)\ar[ru]\\
&P(1)\choose 0\ar[ru]&&0\choose P(4)\ar[ru]}$$

By \cite[Example 2.12]{P}, we have that
$$ \xymatrix{\mod \Lambda'\ar[rr]|{i_{*}}&&\ar@/_1pc/[ll]|{i^{*}}\ar@/^1pc/[ll]|{i^{!}}\mod \Lambda
\ar[rr]|{j^{\ast}}&&\ar@/_1pc/[ll]|{j_{!}}\ar@/^1pc/[ll]|{j_{\ast}}\mod \Lambda''}$$
is a recollement of module categories, where
\begin{align*}
&i^{*}({X\choose Y}_{f})=\Coker f, & i_{*}(X)={X\choose 0},&&i^{!}({X\choose Y}_{f})=X,\\
&j_{!}(Y)={N\otimes_{\Lambda''} Y\choose Y}_{1}, & j^{*}({X\choose Y}_{f})=Y, &&j_{*}(Y)={0\choose Y}.
\end{align*}
\medskip

\begin{itemize}
\item[\rm (1)] Let $T_1=P(1)\oplus S(1)\in \mod \Lambda'$ and $T_3=P(3)\oplus P(4)\oplus P(5)\in \mod \Lambda''$. Then $T_1$ is a cotilting $ \Lambda'$-module and $T_3$ is a $2$-cotilting $\Lambda''$-module. We have two cotorsion pairs:
    \begin{align*}
(\U_1,\V_1)=(\mod \Lambda',\add T_1),\\
(\U_3,\V_3)=(\add T_3,\mod \Lambda'').
\end{align*}
Note that $T_1,T_3$ are just $T'$ and $T''$ given in \cite[Example 4.1(1)]{MZ} respectively. They are also tilting modules, and the tilting cotorsion pairs they induce are still $(\U_1,\V_1)$ and $(\U_3,\V_3)$ respectively. By Theorem \ref{main5} we have a cotilting $\Lambda$-module
$$T={0\choose P(5)}\oplus{S(1)\choose 0}\oplus {P(1)\choose P(3)}\oplus{P(1)\choose P(4)}\oplus {P(1)\choose 0} ,$$
which is different from the tilting $\Lambda$-module
$${0\choose P(5)}\oplus{S(2)\choose P(4)}\oplus {P(1)\choose P(3)}\oplus{P(1)\choose P(4)}\oplus {P(1)\choose 0}$$
got in \cite[Example 4.1(1)]{MZ}.
\medskip

\item[\rm (2)] Let $T_1=P(1)\oplus S(2)\in \mod \Lambda'$ and $T_3=P(3)\oplus P(4)\oplus S(3)\in \mod \Lambda''$. Then $T_1$ is a tilting $\Lambda'$-module and $T_3$ is a $2$-tilting $\Lambda''$-module. We have two cotorsion pairs:
    \begin{align*}
(\U_1,\V_1)=(\add T_1,\mod \Lambda'),\\
(\U_3,\V_3)=(\mod \Lambda'',\add T_3).
\end{align*}
Then we have a $2$-tilting $\Lambda$-module
$$T={S(2)\choose 0}\oplus{S(2)\choose P(4)}\oplus {P(1)\choose 0}\oplus{P(1)\choose P(3)}\oplus {S(1)\choose S(3)}$$
by gluing $T_1$ and $T_3$. We also have
$$i_*(P(1))={P(1)\choose 0},i_*(S(2))={S(2)\choose 0}, j_!(P(3))={P(1)\choose P(3)},j_!(P(4))={S(2)\choose P(4)}, j_!(S(3))={S(1)\choose S(3)}.$$
\end{itemize}

\end{exm}

\bigskip
\bigskip

\section*{Acknowledgments}
The authors would like to thank Tiwei Zhao for the helpful discussions.

\end{document}